\documentclass[review]{elsarticle}

\usepackage{lineno,hyperref}
\usepackage{amssymb,amsmath,amsfonts}
\def\BBox{\kern  -0.2cm\hbox{\vrule width 0.2cm height 0.2cm}}

\def\PP{\mathcal{P}}
\def\LL{\mathcal{L}}
\def\HH{\mathcal{H}}
\usepackage{color}
\usepackage{enumerate}
\usepackage{fancyhdr}
 
\PassOptionsToPackage{pdf}{pstricks} 
\usepackage{pstricks,pst-plot,pst-node,pst-func}
\usepackage{graphicx}
\usepackage{pst-slpe}

\usepackage{pstricks-add}
\usepackage[off]{auto-pst-pdf}
\usepackage{pgf,tikz}
\usetikzlibrary{arrows}
\usepackage{pstricks-add}
\usepackage{pstricks}
\usepackage{pst-node}
\usepackage{graphicx}
\usepackage{pst-all}
\modulolinenumbers[1]

\newtheorem{lemma}{Lemma}[section]
\newtheorem{theorem}{Theorem}[section]
\newtheorem{definition}{Definition}[section]
\newtheorem{corollary}{Corollary}[section]
\newtheorem{proposition}{Proposition}[section]
\newtheorem{remark}{Remark}[section]

\newrgbcolor{grayclaro}{.815 .815 .815}

\journal{Journal of \LaTeX\ Templates}

\begin{document}
	
	\begin{frontmatter}
		
		\title{On the packing chromatic number of Moore graphs}

		\author{
			J. Fres\'an-Figueroa\footnote{jfresan@correo.cua.uam.mx}, D. Gonz\'alez-Moreno\footnote{dgonzalez@correo.cua.uam.mx. This
				research was supported by CONACyT-M\'exico,
				under
				project CB-222104.}%
				 ~and  M. Olsen\footnote{olsen@correo.cua.uam.mx. 
				This research was supported by CONACyT-M\'exico, under project CB-A1-S-45208}
				\\
			{\small Departamento de Matem\'aticas Aplicadas y Sistemas}\\
			{\small Universidad Aut\'onoma Metropolitana - Cuajimalpa, M\'exico}
			}

		\begin{abstract}
			The \emph{packing chromatic number $\chi_\rho (G)$} of a graph $G$ is the smallest integer $k$ for which there exists a vertex coloring $\Gamma: V(G)\rightarrow \{1,2,\dots , k\}$ such that any two vertices of color $i$ are at distance at least $i + 1$.  For $g\in \{6,8,12\}$,  $(q+1,g)$-Moore graphs are $(q+1)$-regular graphs with girth $g$ which are the incidence graphs of a symmetric generalized $g/2$-gons of order $q$. In this paper we study the packing chromatic number of a $(q+1,g)$-Moore graph $G$. For $g=6$ we present the exact value of $\chi_\rho (G)$. For $g=8$, we determine $\chi_\rho (G)$ in terms of the intersection of certain structures in generalized quadrangles. For $g=12$, we present lower and upper bounds for this invariant when $q\ge 9$ an odd prime power.
		\end{abstract}
		
		\begin{keyword}
			\texttt{Packing chromatic number}\sep \texttt{Moore graphs}\sep 
			\texttt{Cages}\sep 
			\texttt{Ovoids}
			\MSC[2010] 05C15\sep 05C35 \sep 51E12
		\end{keyword}
		
	\end{frontmatter}

	\section{Introduction and definitions}

	The concept of packing coloring was introduced, under the name broadcast coloring, by Goddard et al. \cite{GHH08} to solve problems related to the assignment of broadcast frequencies to radio stations. The idea behind packing colorings is that in a network, the signals of two stations  using the same  frequency will interfere unless they are located sufficiently far apart. 
	The term broadcast coloring was renamed as packing chromatic by Bre\v{s}ar et al. \cite{BKR07}.
	A  \emph{packing $k$-coloring}  of  a  graph $G$ is  a  function $\Gamma : V(G)\rightarrow \{1,2,\dots ,k\}$ such that any two vertices of color $i$ are at distance at least $i+1$. The \emph{packing chromatic number} of $G$, $\chi_{p}(G)$,  is the smallest integer $k$ for which $G$ has a  packing $k$-coloring. The packing chromatic number has been studied  for several families of graphs: lattices and grids  \cite{FR10,MR17}, cubic graphs \cite{BK18}, subcubic outerplanar graphs \cite{GH19}, distance graphs \cite{EH12}, sierpinsky graphs \cite{BKR16}, hypercubes \cite{TV15},  cartesian products and trees \cite{BKR07}, among others. Kim et al. \cite{KLM18} showed that determining the packing chromatic number of graphs with diameter at least $3$ is NP-complete. 
	
	All graphs considered in this work are finite, simple and undirected.  We follow the book of Bondy and Murty \cite{BM08} for
	terminology and notations  not defined here. The \emph{distance} between two vertices $u$ and $v$ is denoted by $d(u,v)$.
	The \emph{$m$-neighborhood} of a vertex set $S$, denoted by $N^m(S)$ is the set of vertices at distance $m$ from $S$. If $S=\{v\}$ we denote it by $N^m(v)$ and if $m=1$ we omit the superindex $m$.  
	
	Given two integers $k\ge 2$ and $g\ge 3$ a \emph{$(k,g)$-graph} is a $k$-regular graph of girth $g$.  If $G$ is a $(k,g)$-graph, then 
	\begin{equation*}\label{cagesbound}
	|V(G)|\ge n_0(k,g) =
	\begin{cases}
	\dfrac{2(k-1)^{g/2}-2}{k-2} & \mbox{if $g$ is even};\cr 
	\dfrac{k(k-1)^{(g-1)/2}-2}{k-2} & \mbox{if $g$ is odd.} 
	\end{cases}
	\end{equation*}
	The number $n_0(k,g)$ is called the \emph{Moore bound}. A $(k,g)$-graph of order $n_0(k,g)$ is called a \emph{$(k,g)$-Moore graph} and they  are almost completely characterized \cite{BI73, D73}. 	
	By definition, a $(k,g)$-Moore graph is a $(k,g)$-graph of minimum order. In general, the $(k,g)$-graphs of minimum order are called \emph{$(k,g)$-cages} and have received a lot of attention. For more information on Moore graphs and cages see \cite{EJ13, MS11}.

	In this paper we study the packing chromatic number of $(k,g)$-Moore graphs. In Section  2, we summarize some known results of the packing chromatic number  and use them to obtain this invariant for  Moore graphs with either $k=2$  or  $g\le 5$ and $k\neq 57$. In Section 3, 
 we obtain the exact value of $\chi_\rho (G)$ of Moore graphs with girth $6$. For girth $8$, we determine the packing chromatic number of Moore graphs  in terms of the intersection of certain structures in generalized quadrangles. For $(q+1,12)$-Moore graphs, we present lower and upper bounds for this invariant when $q\ge 9$ is an odd prime power and we use the unofficial  GAP \cite{GAP}  package YAGS \cite{YAGS}  to obtain an upper bound for $q\in \{2,3\}$.

	\section{Previous results} 
	
	We begin this section with  the definition of a symmetric generalized polygon (or $n$-gon). Let $\cal{P}$ (the set of points) and $\cal{L}$ (the set of lines) be disjoint non-empty sets, and let $I$  be the point-line incidence relation. Let ${\cal{I}}=({\cal{P}},{\cal{L}}, {{I}})$, and let $G=G[\PP, \LL]$ be the  bipartite incidence graph on ${\cal{P}}\cup {\cal{L}}$ with edges joining the points from ${\cal{P}}$ to their incident lines in $\cal{L}$. Following notation of  \cite{EJ13}, the ordered triple $\cal{I}$ is  a \emph{symmetric generalized $n$-gon} of order $q$ subject to the following regularity conditions.
	\begin{itemize}
		\item[GP1:] There exists an integer $q\ge 1$ such that every line is incident to exactly $q+1$ points and every point is incident to exactly $q+1$ lines.
		
		\item[GP2:] Any two distinct lines intersect in at most one point and there is at most one line through any two distinct points.
		
		\item[GP3:] The incidence graph $G=G[{\cal{P}},{\cal{L}}]$ has diameter  $n$ and  girth  $2n$.
	\end{itemize}
	
	For more information on symmetric generalized $g/2$-gons see  \cite{Libro Rojo,M98}. The following theorem summarize the partial characterization of Moore graphs.
	
		\begin{theorem}\label{Moore}\cite{BI73, D73}
			There exists a Moore graph  of degree $k$ and girth $g$ if and only if
			any of the following conditions hold
			\begin{enumerate} 
				\item $k=2$ and $g\ge 3$ (cycles);
				\item $g=3$ and $k\ge 2$ (complete graphs); 
				\item $g=4$ and $k\ge 2$ (complete bipartite graphs);
				\item $g=5$, and $k=3$ (Petersen graph) or $k=7$ (Hoffman-Singleton graph),
				or possibly $k=57$;
				\item $g\in \{6,8,12\}$ and there exists a symmetric generalized $g/2$-gon of order
				$k-1$ (incidence graphs of  symmetric generalized $g/2$-gons).
			\end{enumerate}
		\end{theorem}
		
	The packing chromatic number has already been determined for some Moore graphs \cite{GHH08} such as complete bipartite graphs, complete graphs and cycles. 
	Let $\beta(G)$ be the independence number of a graph $G$. We rewrite a result of Goddard et al. in terms of the independence number. 
	\begin{proposition}\cite{GHH08}\label{propo1}
		For every graph $G$ of order $n$ and diameter two, 
		$\chi_\rho (G)=  n-\beta(G) + 1$.
	\end{proposition}
	Let $P$ denote the Petersen graph and let $HS$ denote the Hoffmann-Singleton graph, recall that  $diam(P)=diam(HS)=2$. Using Proposition \ref{propo1} and results obtained by Goddard et al. \cite{GHH08}, we obtain the following remark.
	\begin{remark} Let $n\ge 1$  and $m\ge 3$ be integers. Then:
		$\chi_\rho (K_n)=n$, $\chi_\rho (K_{n,n})=n+1$, $\chi_\rho (C_m)=3$ if $m\equiv 0~ (mod~ 4)$ and $\chi_\rho (C_m)=4$ if $m\not\equiv 0~(mod~ 4)$,  $\chi_\rho (P)=7$ and $\chi_\rho (HS)=36$.
	\end{remark}

Hence, the for items 1., 2., 3. and 4. of Theorem \ref{Moore} the packing chromatic number is known except for the possibly existing $(57,5)$-Moore graph. 		
	\section{The packing chromatic number of $(q+1,g)$-Moore graphs for $g\in \{6,8,12\}$}
	
	In this section we study the packing chromatic number for graphs considered in  item 5. of Theorem \ref{Moore}.
	 For $g\in \{6,8,12\}$, $(q+1,g)$-Moore graphs  are incidence graphs of symmetric generalized $g/2$-gons and thus $\beta(G)=n/2$. We use the following proposition.
 	\begin{proposition}\cite{GHH08}\label{cota beta}
	 		If $G$ is a bipartite graph of order $n$ and diameter $3$, then
	 		$$n-\beta(G)\leq\chi_\rho (G)\leq n-\beta(G)  +1.$$
	 	\end{proposition}
		The following lemma is straightforward, but for sake of completeness we include the proof.   
		\begin{lemma}\label{independiente}
			Let $G$ be a $(q+1,g)$-Moore graph of $n$ vertices with vertex partition $(\mathcal{P}, \mathcal{L})$. If $S$ is a maximum independent set, then either $S= \mathcal{P}$ or $S=\mathcal{L}$. 
		\end{lemma}
		\begin{proof}
			Let $P_S=S\cap {\cal{P}}$ and $L_S=S\cap {\cal{L}}$. Suppose for a contradiction, that  $L_S\neq \emptyset \neq {P_S}$. Let $|{L_S}|=r$, let $|{P_S}|=n/2-r$ and 
			let $G'$ be the bipartite subgraph induced by $L_S\cup ( {\cal{P}}\setminus P_S)$ of order $2r$. Since $S$ is an maximum independent set,  $N_G(L_S)\subset ({\cal{P}}\setminus P_S)$,  $G'$ has size $r(q+1)$ and it is $(q+1)$-regular. Therefore, $G'$ is a connected component of $G$, contradicting that $G$ is connected.
		\end{proof}
		\begin{theorem}
			If $G$ is a $(q+1,6)$-Moore  graph, then
			$$\chi_\rho (G)= q^2+q+2.$$
		\end{theorem}
		\begin{proof}
			Let $G$ be a  $(q+1,6)$-Moore graph of order $n=n_0(q+1,6)=2(q^2+q+1)$.  Since $G$ is the incidence graph of a symmetric generalized $3$-gon (projective plane), it follows that the diameter of $G$ is $3$, $\beta(G)=n/2$ and there is a bipartition $({\cal{P}},{\cal{L}})$ of $V(G)$ such that $|{\cal{P}}|=|{\cal{L}}|=n/2$. By Proposition \ref{cota beta},  $n/2 \leq\chi_\rho (G)\leq n/2 +1$. 
			Suppose, for the sake of contradiction, that $\chi_\rho (G)=n/2$. Let $\Gamma:V(G)\rightarrow \{1,2,\dots ,n/2\}$ be an optimal packing coloring of $G$. Since diameter of $G$ is $3$,  it follows that $\Gamma^{-1}(i)$ is a singular class for $3\le i\le n/2$, $\Gamma$ has at least $n/2-2$ singular classes and exactly $n/2+2$ vertices of color $1$ or $2$. If $\Gamma^{-1}(2)$ is a singular class, then $|\Gamma^{-1}(1)|=n/2+1> \beta (G)$ contradicting the fact that the vertices of color $1$ induce an independent set.
			Hence, $\Gamma^{-1}(2)$ is not a singular class. Let $u,v\in \Gamma^{-1}(2)$, then $d(u,v)= 3$, implying that $u$ and $v$ belong to different parts of the bipartition. Therefore, $|\Gamma^{-1}(2)|=2$,  $|\Gamma^{-1}(1)|=n/2$ and $\Gamma^{-1}(1)\cap \mathcal{P} \neq \emptyset \neq \Gamma^{-1}(1)\cap \mathcal{P}$ which contradicts Lemma \ref{independiente}. Thus, $\chi_\rho (G)=n/2+1$ and the result follows.
		\end{proof}
		
		Two points or two  lines of a symmetric generalized $g/2$-gon  of order $q$ are called \emph{opposite} if they are at distance $g/2$. An \emph{ovoid} $\cal{O}$ (resp. \emph{spread} $\cal{S}$) in a generalized $g/2$-gon of order $q$ is a set of $q^{g/4}+1$ mutually opposite points (lines), such that every element $x$ of the generalized $g/2$-gon is at distance at most $g/4$ from at least one point $p$ of $\mathcal{O}$ (resp. one line $l$ of $\cal{S}$).	
	Ovoids and spreads have similar properties and have been widely studied, for more information see \cite{M98}. In graph terminology, an ovoid of a $(q+1,g)$-Moore graph for $g\in \{6,8,12\}$  is a set  $q^{g/4}+1$ of vertices which are mutually at distance $g/2$.
	\begin{proposition}\label{lemacamino} \cite{Birregular}
		A $(q+1,g)$-Moore graph with $q$ a prime power and $g= 8$, or $q$ an odd prime power different
		from 5 and 7 and $g =12$, contain exactly $q^{g/4}+1$ vertices which are mutually at distance  $g/2$.
	\end{proposition}

	We use the following  coordinatization of a $(q+1,8)$-Moore graph, where $V_0=\cal{P}$ and $V_1 = \cal{L}$.
	\begin{definition}\cite{AABL}\label{cordinadas Camino}
		Let $\mathbb{F}_q$ be a finite field with $q \ge 2$ a prime power and $\rho$ a symbol
		not belonging to $\mathbb{F}_q$. Let $G$ be the incidence graph of a generalized quadrangle of order $q$. Let $(V_0, V_1)$ be the bipartition of $G$ with $V_i =
		\mathbb{F}^3_q\cup \{(\rho, b, c)_i, (\rho, \rho, c)_i : b, c \in\mathbb{F}_q\} \cup \{(\rho, \rho, \rho)_i\}$, $i \in \{0, 1\}$ and edge set defined as follows:
		
		For all $a \in\mathbb{F}_q \cup\{\rho\}$ and for all $b, c \in\mathbb{F}_q$ :
		
		$N_{G} ((a, b, c)_1) =
		\left\{\begin{array}{ll}
		\{(w, aw + b, a^2w + 2ab + c)_0 : w \in\mathbb{F}_q\} \cup \{(\rho, a, c)_0\}& if~ a \in\mathbb{F}_q; \\
		\{(c, b,w)_0 : w \in\mathbb{F}_q\} \cup \{(\rho, \rho, c)_0\} & if~ a = \rho.
		\end{array}\right. $
		
		$N_{G} ((\rho, \rho, c)_1) = \{(\rho, c,w)_0 : w \in\mathbb{F}_q\} \cup \{(\rho, \rho, \rho)_0\}$
		
		$N_{G} ((\rho, \rho, \rho)_1) = \{(\rho, \rho,w)_0 : w \in\mathbb{F}_q\} \cup \{(\rho, \rho, \rho)_0\}$.
		
		Or equivalently, for all $i \in\mathbb{F}_q \cup \{\rho\}$ and for all $j, k \in\mathbb{F}_q$ :
		
		$N_{G} ((i, j, k)_0) =
		\left\{\begin{array}{ll}
		\{(w, j - wi, w^2i - 2wj + k)_1 : w \in\mathbb{F}_q\} \cup\{(\rho, j, i)_1\}& if~ i \in\mathbb{F}_q;\\
		\{(j,w, k)_1 : w \in\mathbb{F}_q\} \cup \{(\rho, \rho, j)_1\}&  if~ i = \rho.
		\end{array}\right. $
		
		$N_{G} ((\rho, \rho, k)_0) = \{(\rho,w, k)_1 : w \in\mathbb{F}_q\} \cup \{(\rho, \rho, \rho)_1\};$
		
		$N_{G} ((\rho, \rho, \rho)_0) = \{(\rho, \rho,w)_1 : w \in\mathbb{F}_q\} \cup \{(\rho, \rho, \rho)_1\}$.
	\end{definition}
	
	Recall that a $(q+1,8)$-Moore graph  $G$ has order  $2(q+1)(q^2+1)$,  diameter $4$, an ovoid of $G$ has cardinality $q^2+1$  and a $(q+1,8)$-Moore graph exist when $q$ is a prime power. 

	\begin{theorem}
		Let $q$ be an odd prime power. If $G$ is a $(q+1,8)$-Moore graph, then
		$$(q^2+1)(q-1)+3\le \chi_\rho (G)\le (q^2+1)(q-1)+4.$$
	\end{theorem}
	\begin{proof}
		Let $G=G[V_0,V_1]$ be a $(q+1,8)$-Moore graph. For the upper bound, we use the cordinatization in Definition \ref{cordinadas Camino} to construct a packing coloring of $G$ with $(q^2+1)(q-1)+4$ colors.
		We prove that the sets $O=\{(\rho,\rho,\rho)_0\}\cup\{(i,j,0)_0\mid i,j\in \mathbb{F}_q\}$ and $O'=\{(\rho,\rho,\rho)_0\}\cup\{(i,j,1)_0\mid i,j\in \mathbb{F}_q\}$ are ovoids in $G$.
			Let $x,y\in O$ be two distinct vertices. Since $x,y\in V_0$, the distance between them is even. 
			We prove that $d(x,y)=4$. 
			By  Definition \ref{cordinadas Camino}, $d((\rho,\rho,\rho)_0,(i,j,0)_0)=4$. Suppose by contradiction that there exist $x,y\in O$ such that  $d(x,y)=2$. Let $x=(i,j,0)_0$, let $y=(i',j',0)_0$ and let $L=(a,b,c)_1$ be the unique vertex adjacent to $x$ and $y$. Then, by  Definition \ref{cordinadas Camino},
			$$(i,j,0)_0=(i,ai+b,a^2i+2ab+c)_0\qquad and\qquad (i',j',0)_0=(i',ai'+b,a^2i'+2ab+c)_0$$
			Since $a^2i+2ab+c=0=a^2i'+2ab+c$, it follows that $i=i'$, $j=ai+b=ai'+b=j'$ and $x=y$, a contradiction. Hence, every pair of vertices in $O$ are at distance $4$. Furthermore,  $i,j\in \mathbb{F}_q$ hence $|O|=q^2+1$ and $O$ is an ovoid. Analogously $O'$ is an ovoid and by Definition \ref{cordinadas Camino}, $O\cap O'=\{(\rho,\rho,\rho)_0\}$.
		We define $\Gamma:V(G)\rightarrow \{1,2, \dots , (q^2+1)(q-1)+4\}$ as follows: $\Gamma^{-1}(1)=V_1$, $\Gamma^{-1}(2)=O$, $\Gamma^{-1}(3)=O'\setminus \{(\rho,\rho,\rho)_0\}$, and color each vertex of the remaining vertices with a different color. Thus, $\Gamma$ is a packing coloring that uses $2(q^2+1)(q+1)-(q^2+1)(q+1)-(q^2+1)-q^2+3=(q^2+1)(q-1)+4$ colors and $$\chi_\rho (G)\le (q^2+1)(q-1)+4.$$

		For the lower bound, we use geometric properties of the generalized quadrangles. Recall that $V_0 = \mathcal{P}$ and $V_1 =\mathcal{L}$.  
		Observe that  
		$\Gamma^{-1}(i)$ is a singular class for $i\ge 4$, and   $\Gamma^{-1}(3)$ is contained in the same part of $G$, because vertices of color $3$ are at distance $4$ and $diam(G)=4$.
		Let $\beta_4$ be the maximal order of a set of vertices at distance at least $4$. By Proposition \ref{lemacamino}, $\beta_4= q^2+1$ and thus $|\Gamma^{-1}(3)|\le  q^2+1$.
		Since $\beta(G)= (q^2+1)(q+1)$, it follows that $|\Gamma^{-1}(1)|\le  (q^2+1)(q+1)$. Hence, it suffices to prove that $|\Gamma^{-1}(2)|\le  q^2+1$.
		If $\Gamma^{-1}(2)\subset {\cal{P}}$, then $\Gamma^{-1}(2)$ is a set of vertices at distance 4, and by Proposition \ref{lemacamino}, $|\Gamma^{-1}(2)|\le q^2+1$ and the result holds. The case $\Gamma^{-1}(2)\subset {\cal{L}}$ is analogous.
		Assume that $\Gamma^{-1}(2)\cap {\cal{P}}\neq \emptyset\neq \Gamma^{-1}(2)\cap {\cal{L}}$. Let ${\cal{P}}_2 =\Gamma^{-1}(2)\cap {\cal{P}}$  and ${\cal{L}}_2=\Gamma^{-1}(2)\cap {\cal{L}}$. 
		If $|{\cal{P}}_2|=q^2$, then ${\cal{P}}_2$ can be extended to an ovoid (Theorem 2.7.1  \cite{Libro Rojo}). Let $v\in {\cal{P}}$ such that ${\cal{P}}_2 \cup \{v\}$ is an ovoid. By the  definition of ovoids, every vertex in ${\cal{L}}$ is adjacent to exactly one vertex in ${\cal{P}}_2\cup \{v\}$. Note that vertices in ${\cal{L}}_2$ can not be adjacent to vertices in ${\cal{P}}_2$,  thus every vertex in ${\cal{L}}_2$ is adjacent to $v$. Since   vertices in $\Gamma^{-1}(2)$ are pairwise at distance at least $3$, therefore $|{\cal{L}}_2|=1$ and $|\Gamma^{-1}(2)|=q^2+1$.
				
		Assume that $|{\cal{P}}_2|	\le q^2-1$. Let $|{\cal{P}}_2|=q^2-\delta$ with $\delta \ge 1$. Since vertices in $\Gamma^{-1}(2)$ are pairwise at distance at least $3$, the vertices of ${\cal{P}}_2$ do not have neighbours in common, therefore $|N({\cal{P}}_2)|=(q^2-\delta)(q+1)$. Let $S={\cal{L}}\setminus N({\cal{P}}_2)$. Observe that ${\cal{L}}_2$ is contained in $S$. Since $|{\cal{L}}| = (q^2+1)(q+1)$, it follows that  $|S|= (q+1)(\delta+1)$. Let $l\in S$ and let $N(l)=\{ y_0,y_1,\dots,y_q\}$. Notice that $y_i \notin {\cal{P}}_2$ for $i \in\{0, 1 , \ldots ,q\}$.  Let $z_i\in N(y_i)$ and let $z_j\in N(y_j)$ with $z_i \neq l \neq z_j$. Then, $z_i\neq z_j$, otherwise $(l,y_i,z_i,y_j,l)$ is a $4$-cycle, a contradiction. Moreover,  $z_i$ and $z_j$ do not share neighbours, otherwise  $(l,y_i,z_i,p,z_j,y_j,l)$ is $6$-cycle, a contradiction.
		Since the vertices of ${\cal{P}}_2$  do not have neighbours in common, the neighbours of ${\cal{P}}_2$ induce a partition of $\LL-S$ into $q^2-\delta$ classes, $U_{v}$ for $v\in {\cal{P}}_2$. In order to prove that each class $U_v$ has at most one neighbour in $N(l)$,
		we assume for a contradiction, that there is a class $U_\alpha$ such that $N(U_\alpha)$ contains at least two vertices $y_i,y_j\in N(l)$. In this case, either $(x,y_i,l,y_j,x)$ is a $4$-cycle for some $x\in U_\alpha$ or  $(\alpha, x_1,y_i,l,y_j,x_2,\alpha)$ is a $6$-cycle for some $x_1,x_2\in U_\alpha$, a contradiction. Hence,   each class $U_v$ has at most one neighbour in $N(l)$, see Figure \ref{Fig8}. 
		
		\begin{figure}[h!]
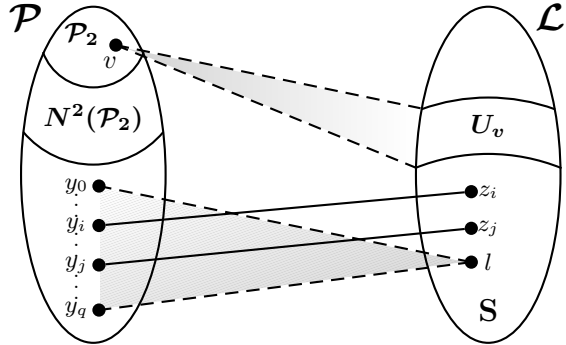
	
		\pspicture(-9.1,-1.9)(6.6,2.4)
		\centering
		\SpecialCoor
		\psset{linecolor=black, dotsize=5pt, unit=.75}
		\rput(2,0){

			\pspolygon[ linecolor=white, fillstyle=slope,slopeangle=-70,
			slopebegin=white,
			slopeend=grayclaro](-.3,-1.55)(-6.9,-.2) (-6.9,-2.4)
			
			\pspolygon[ linecolor=white, fillstyle=slope,slopeangle=0,
			slopebegin=grayclaro,
			slopeend=white] (-6.6,2.3)(-1.27,0.12)(-1.2,1.18)
			\psellipse(0,0)(1.3,3)
			
			\psline[linestyle=dashed](-.3,-1.55)(-6.9,-.2)
			\psline[linestyle=dashed](-.3,-1.55)(-6.9,-2.4)
			
			\rput(-7,0){
				\psellipse(-0,0)(1.3,3)
				
				\psarc(0,2.5){.95}{200}{340}
				\psarc(0,1.8){1.62}{220}{320}
				\rput(-.2,2.5){$ \boldsymbol{{\cal{P}}_{2}} $}
				\rput(.3,2){$v$}
				\dotnode(.4,2.3){v}
				\rput(0,1.){$\boldsymbol{ N^{2}({\cal{P}}_{2}) }$}
				\rput(-.3,-.2){\small$y_0$}
				\rput(-.3,-.4){\tiny \textbf{$\vdots$}}
				\rput(-.3,-.9){\small$y_i$}
				\rput(-.3,-1.1){\tiny \textbf{$\vdots$}}
				\rput(-.3,-1.6){\small$y_j$}
				\rput(-.3,-1.85){\tiny \textbf{$\vdots$}}
				\rput(-.3,-2.4){\small$y_q$}
				
				\dotnode(.1,-.2){y0}
				\dotnode(.1,-.9){yi}
				\dotnode(.1,-1.6){yj}
				\dotnode(.1,-2.4){yq}	
			}
			
			\dotnode(-.3,-1.55){l}
			\rput (0,-1.55){\small$l$}
			\rput (0,-2.4){$\large\textbf{S}$}
			\dotnode(-.3,-.95){zj}
			\rput (0,-.95){\small$z_j$}
			\dotnode(-.3,-0.3){zi}
			\rput (0,-0.3){\small$z_i$}
			
			\ncline{yi}{zi}
			\ncline{yj}{zj}
			
			\psarc(0,-3.1){3.47}{68}{112}			
			\psarc(0,-2.1){3.47}{70}{110}
			\rput(0,.87){$\boldsymbol{U_v}$}
			
			\rput(-8.2,2.8){\Large$\boldsymbol{{\cal{P}}}$}
			\rput(1.1,2.8){\Large$\boldsymbol{{\cal{L}}}$}
			
			\psline[linestyle=dashed](-6.6,2.3)(-1.2,1.18)
			\psline[linestyle=dashed](-6.6,2.3)(-1.27,0.12)
		}

		\endpspicture
		\caption{Case $|{\cal{P}}_2|\le q^2-1$ and ${\cal{L}}_2\neq \emptyset$.}
		\label{Fig8}
		\end{figure}
		Let $t_i=d_{S-l}(y_i)$. Since $g=8$,  $N(u)\cap N(v) = \{l\}$ for any pair of vertices  $u, v \in N(l)$. For each $v\in {\cal{P}}_2$, $U_v$ has at most one neighbour in $N(l)$,  implying that  $|N^2_{S-l}(l)|=\sum_{i=0}^q t_i \geq (q+1)q-(q^2- \delta)=q+\delta$.  Since $l$ has at least  $q+\delta$ vertices in $S-l$ at distance 2 and  ${\cal{L}}_2\subset S$, it follows that if $|{\cal{L}}_2|=k$, then $k+k(q+\delta) \le |{\cal{L}}_2 |+ |N_{S-l}^2({\cal{L}}_2)|\le |S|= (q+1)(\delta+1)$. Therefore,
		$$
		k\le \frac{(q+1)(\delta+1)}{q+\delta+1}.
		$$
		Hence, 
		$$
		|\Gamma^{-1}(2)|=k+q^2-\delta\le q^2-\delta+\frac{(q+1)(\delta+1)}{q+\delta+1}=q^2+\frac{q+1-\delta^2}{q+1+\delta}<q^2+1.
		$$
		
		Thus, $|\Gamma^{-1}(2)| \leq q^2 +1$ which implies $(q^2+1)(q-1)+3\le \chi_\rho (G)$ and the result holds.
	\end{proof}
	
	\begin{corollary}
		Let $Q$ be a generalized quadrangle of order $q$ and let $G$ be the incidence graph of $Q$. Then $Q$ contains two disjoint ovoids or two disjoint spreads if and only if $\chi_\rho (G)=(q^2+1)(q-1)+3$.
	\end{corollary}
	
	\begin{proof}
		Let $G=G[\PP, \LL]$  be the incidence graph of a  generalized quadrangle $Q$. If $Q$ contains two disjoint ovoids (resp. spreads), then we can color the vertices of $\mathcal{L}$ (resp. $\mathcal{P}$) with color $1$, the vertices of each ovoid (resp. spread) with colors $2$ and $3$, respectively, and the $(q^2+1)(q-1)$ remaining vertices with different colors. Therefore, $\chi_\rho (G)\le(q^2+1)(q-1)+3$  and the equality is attained.
		
		If $\chi_\rho (G)=(q^2+1)(q-1)+3$, then let $\Gamma$ be a  packing coloring attaining this number. Since $diam(G) =4$, $|\Gamma^{-1}(i)|=1$ for $i\geq 4$. Also, $|\Gamma^{-1}(1)|\leq q^3 + q^2 + q+ 1$ , $|\Gamma^{-1}(2)|\leq q^2 + 1 $ and $|\Gamma^{-1}(3)|\leq   q^2 + 1$. Since we have $(q^2+1)(q-1)$ singular classes, to achieve $\chi_\rho (G)$ the equality must hold in the previous inequalities. By Lemma \ref{independiente}, $\Gamma^{-1}(1)$ is either $\mathcal{P}$ or $\mathcal{L}$.  Therefore $\Gamma^{-1}(2)$ and $\Gamma^{-1}(3)$ must be disjoint ovoids or two disjoint spreads.
	\end{proof}
	
	For $q$ an odd prime power, every pair of ovoids intersects \cite{BS89, Barlotti}. For $q$ an even prime power, it is conjectured that every pair of ovoids intersects, but as far as the authors know, a proof is only known for  $2 \leq q \leq 64$. 
	
	\begin{corollary}
		If $G$ is a $(q+1,8)$-Moore graph with $q$ an odd prime power or $q \leq 64$ an even prime power, then
		$$\chi_\rho (G)= (q^2+1)(q-1)+4.$$
	\end{corollary}
	
We continue studying the packing chromatic number of $(q+1,12)$-Moore graphs which are incidence graphs of generalized hexagons. First we establish upper bounds for some chromatic classes.	
		
	\begin{lemma}\label{clase4}
		Let $\Gamma$ be a packing coloring of a $(q+1,12)$-Moore graph with $q$ an odd prime power different from 5 and 7. Then $|\Gamma^{-1}(4)|\le 2q^3-2q^2+2q$.
	\end{lemma}
	\begin{proof}
		Let $\HH$ be a generalized hexagon of order $q$. Let $G=G[\PP, \LL]$ be the  incidence graph of $\HH$. For  every $u,v\in \Gamma^{-1}(4)$,  $d(u,v)\ge 5$, then by Proposition \ref{lemacamino},
		$|\Gamma^{-1}(4)\cap {\cal{L}}|,|\Gamma^{-1}(4)\cap {\cal{P}}|\le q^3+1$. If $|\Gamma^{-1}(4)\cap {\cal{P}}|=q^3+1$, then $\Gamma^{-1}(4)\cap {\cal{P}}$ is an ovoid  and $|\Gamma^{-1}(4)\cap {\cal{L}}|=0$. Therefore $|\Gamma^{-1}(4)|=q^3+1$ and the result follows.
		Similarly if $|\Gamma^{-1}(4)\cap {\cal{L}}|=q^3+1$, then $\Gamma^{-1}(4)\cap {\cal{L}}$ is a spread, and the result holds.
		Let $P_4=\Gamma^{-1}(4)\cap {\cal{P}}$ and let $L_4=\Gamma^{-1}(4)\cap {\cal{L}}$.
		Assume that $|P_4|\geq|L_4|\ge 1$, and let $ r \geq0$ be an integer such that $|P_4|=q^3-r$. For $u,v\in P_4$, $d(u,v)=6$ and $N^2(u)\cap N^2(v)=\emptyset$. Hence 
	\begin{equation*}
		|N^2(P_4)|=(q+1)(q^3-r)q=q^5+q^4-r(q^2+q).
	\end{equation*}
		Let $P^{'}={\cal{P}}\setminus (P_4\cup N^2(P_4))$. Hence 
		$|P^{'}|=(r+1)(q^2+q+1)$. Since $d(u,v)=5$, for every $u\in P_4$ and $v\in L_4$, it follows that  $N(L_4)\subset P^{'}$ and $(q+1)|L_4|\le (r+1)(q^2+q+1)$. Therefore 
		$$
		|L_4|\le (r+1)q+\frac{r+1}{q+1}.
		$$
		Since   $|L_4|\le |P_4|=q^3-r$, then 
		$$
		|L_4|\le \min \left\{q^3-r,(r+1)q+\frac{r+1}{q+1}\right\}.
		$$
		Let $f(q)=(q^4+q^3-q^2-q-1)/(q^2+2q+2)$. Observe that if $r\ge f(q)$, then
		$|L_4|\le q^3-r$ and   $|\Gamma^{-1}(4)|=|P_4|+|L_4|\le 2q^3-2f(q)$.
	If  $r< f(q)$, then 
			$$
			|L_4|<r\left(q+\frac{1}{q+1}\right)+q+\frac{1}{q+1}
			$$
			and
			\begin{eqnarray*}
				|\Gamma^{-1}(4)|& = &|P_4|+|L_4|<q^3-r +r\left(q+\frac{1}{q+1}\right)+q+\frac{1}{q+1}\\
				& = & q^3+q+\frac{1}{q+1}+r\left(q-1+\frac{1}{q+1}\right)\\
				& = & 2q^3-2\frac{q^4-q^3+q^2-q-q}{q^4+2q^2+2}\\
				& = & 2q^3-2f(q).
			\end{eqnarray*}
		In both cases $|\Gamma^{-1}|(4)|\le 2q^3-2f(q)$. Since $f(q)\le q^2-q$ for $q\ge 2$, the result follows. The case when $|L_4|\geq |P_4|$ is proved similarly. 	\end{proof}
	
	A \emph{distance-$2$ ovoid} of a generalized $n$-gon is a subset of the point set with the property that every line contains exactly one point of that subset.  Dual notion is that of a \emph{distance-$2$ spread}. From this definition it follows that a distance-$2$ ovoid or distance-$2$ spread of a generalized hexagon has $q^4+q^2+1$ elements.

	\begin{lemma}\label{clase3}
		Let $\Gamma$ be a packing coloring of a $(q+1,12)$-Moore graph. Then $|\Gamma^{-1}(3)|\le q^4+q^2+1$ and the equality holds if $\Gamma^{-1}(3)$ is a distance-$2$  ovoid or a distance-$2$ spread.
	\end{lemma}
	\begin{proof}
		Let $\HH$ be a generalized hexagon of order $q$. Let $G=G[\PP, \LL]$ be the  incidence graph of $\HH$. For every $u,v\in \Gamma^{-1}(3)$, $d(u,v)\ge 4$.
		Let $P_3=\Gamma^{-1}(3)\cap {\cal{P}}$ and let $L_3=\Gamma^{-1}(3)\cap {\cal{L}}$.
		Observe that 
		$|P_3|+|N^2(P_3)|+|N(L_3)|\le |{\cal{P}}|$ and $|L_3|+|N^2(L_3)|+|N(P_3)|\le |{\cal{L}}|$.
		Therefore
		\begin{equation}\label{ecuacionclase3}
		|P_3|+|L_3|+|N^2(P_3)|+|N^2(L_3)|+|N(P_3)|+|N(L_3)|\le |{\cal{P}}|  +|{\cal{L}}|=|V(G)|.
		\end{equation}

		Let $B$ be the bipartite graph constructed as follows: $V(B)=P_3\cup N^2(P_3)$ and for $x\in P_3$ and $y\in N^2(P_3)$, $xy\in E(B)$ if and only if $y\in N^2(x)$.	
		Observe that for every $x\in P_3$, $d_B(x)=q(q+1)$. On the other hand, let $y\in N^2(P_3)$. If $d_B(y)\geq q+2$, since $d_G(y)=q+1$ there exists at least two vertices $x_1, x_2 \in P_3\cap N^2(y)$ such that $d(x_1,x_2)=2$, a contradiction. Hence $d_B(y)\le q+1$ and
		$$
		|E(B)|=|P_3|q(q+1)=\sum_{y\in N^2(P_3)}d(y)\le |N^2(P_3)|(q+1),
		$$
		implying that $|N^2(P_3)|\ge |P_3|q$.	Analogously, $|N^2(L_3)|\ge |L_3|q$.
		Since $G$ is $(q+1)$-regular, $|N(P_3)|=|P_3|(q+1)$ and $|N(L_3)|=|L_3|(q+1)$.		By (\ref{ecuacionclase3}), 
		\begin{eqnarray*}
			|V(G)| & \ge & |P_3|+|L_3|+|N^2(P_3)|+|N^2(L_3)|+|N(P_3)|+|N(L_3)|\\
			& \ge & |P_3|+|L_3|+(|P_3|+|L_3|)q+(|P_3|+|L_3|)(q+1)\\
			& = & (|P_3|+|L_3|)(2q+2).
		\end{eqnarray*}
		Therefore $|P_3|+|L_3|\le q^4+q^2+1$.
	\end{proof}

	\begin{lemma}\label{clase2}
		Let $\Gamma$ be a packing coloring of a $(q+1,12)$-Moore graph. Then $|\Gamma^{-1}(2)|\le 2\dfrac{q+1}{q+2}(q^4+q^2+1)$. 
	\end{lemma}
	\begin{proof}
		Let $\HH$ be a generalized hexagon of order $q$. Let $G=G[\PP, \LL]$ be the  incidence graph of $\HH$. For every $u,v\in \Gamma^{-1}(2)$, $d(u,v)\ge 3$.
		Let $P_2=\Gamma^{-1}(2)\cap {\cal{P}}$ and let $L_2=\Gamma^{-1}(2)\cap {\cal{L}}$.
		Observe that 
		$|P_2|+|N(L_2)|\le |{\cal{P}}|$ and $|L_2|+|N(P_2)|\le |{\cal{L}}|$.
		Therefore
		\begin{equation}\label{ecuacionclase2}
		|P_2|+|L_2|+|N(P_3)|+|N(L_3)|\le |{\cal{P}}|  +|{\cal{L}}|=|V(G)|.
		\end{equation}
	Since $G$ is $(q+1)$-regular,  $|N(P_2)|=|P_2|(q+1)$ and $|N(L_2)|=|L_2|(q+1)$.		By (\ref{ecuacionclase2}),
		\begin{eqnarray*}
			|V(G)| & \ge & |P_2|+|L_2|+|N(P_2)|+|N(L_2)|\\
			& \ge & |P_2|+|L_2|+(|P_2|+|L_2|)(q+1)\\
			& = & (|P_3|+|L_3|)(q+2).
		\end{eqnarray*}
		Therefore,
		$$
		|\Gamma^{-1}(2)|=|P_2|+|L_2|\le 2\dfrac{q+1}{q+2}(q^4+q^2+1).
		$$	\end{proof}
	
	\begin{theorem}\label{hexagons}
		Let $G$ be a $(q+1,12)$-Moore graph with $q$ an odd prime power $q\ge 9$, then
		$$
		q^5-2q^4-4q^2+9q-18+\frac{42}{q+2}\le \chi_\rho (G)\le q^5+q^4-2q^3-q^2+4.
		$$ 
	\end{theorem}
	\begin{proof}
		Let $q\geq 9$ be an odd prime power,  let $G$ be a  $(q+1,12)$-Moore graph and let $\Gamma$ be a packing coloring of $G$. Since $diam(G)=6$, $\Gamma^{-1}(i)$ is singular, for $i\ge 6$.
		For the lower bound observe that by  Proposition \ref{lemacamino}, $\Gamma^{-1}(5)\le q^3+1$. Let $S$ be the set of singular classes of $\Gamma$, then
	$$
			|S| \ge  |V(G)|-\sum_{i=1}^5|\Gamma^{-1}(i)|.
		$$
		Using that $\Gamma^{-1}(1)$ is an independent set and  Lemmas \ref{clase4}, \ref{clase3} and \ref{clase2}, it follows that
		$$
		|S|\ge q^5-2q^4-4q^2+9q-23-\frac{42}{q+2}.
		$$
		Therefore, $\chi_\rho (G)\ge |S|+5$ and the lower bound holds.
		
		For the upper bound we give a packing chromatic coloring of a  $(q+1,12)$-Moore graph.
		Consider an ovoid $\cal{O}$  of $G$ (which exist by Proposition \ref{lemacamino}) and assign color $5$ to the vertices of $\cal{O}$. Let $x\in {\cal{O}}$ and let $y\in N(x)$.
		Let $T$ be the spanning tree such that every vertex of $T$ is at distance at most 
		$(g-2)/2$ from the edge $xy$, where $x$ and $y$ are labeled by $(0,0)_0$ and $(0,0)_1$, respectively.  The vertex set of $T$ is labeled by $(i,j)_k$ where $k\in \mathbb{Z}_2$ and the subindices are taken in $\mathbb{Z}_2$, $j$ denotes the distance between the edge $xy$ and the vertex $(i,j)_k$ and $0\le i \le q^{j}-1$. The adjacencies of $(0,0)_k$ in $T$ are $(i,1 )_{k+1}$ for $0 \le i < q$ and $(0,0)_{k+1}$. The adjacencies of $(i,j)_k$ in $T$ are the vertices labelled by 
		$(iq+r,j+1)_{k+1}$ where $0\le r < q$ and the vertex labelled by $(i',j-1)_{k+1}$ where $i=i'q+r$ and $0\le r < q$. 
		For $i\in \{0,1,\dots , q^{\alpha}-1\}$, let $B^{\alpha}_i=\{(iq+r,\alpha)_0 : 0\le r<q\}$. Observe that ${\cal{O}}\subseteq (0,0)_0\bigcup \left(\cup^{q^4-1}_{i=0}B^5_1\right)$. A general scheme of such labeled tree can be seen in Figure  \ref{FigArbol}. 
		
\begin{figure}[h!]
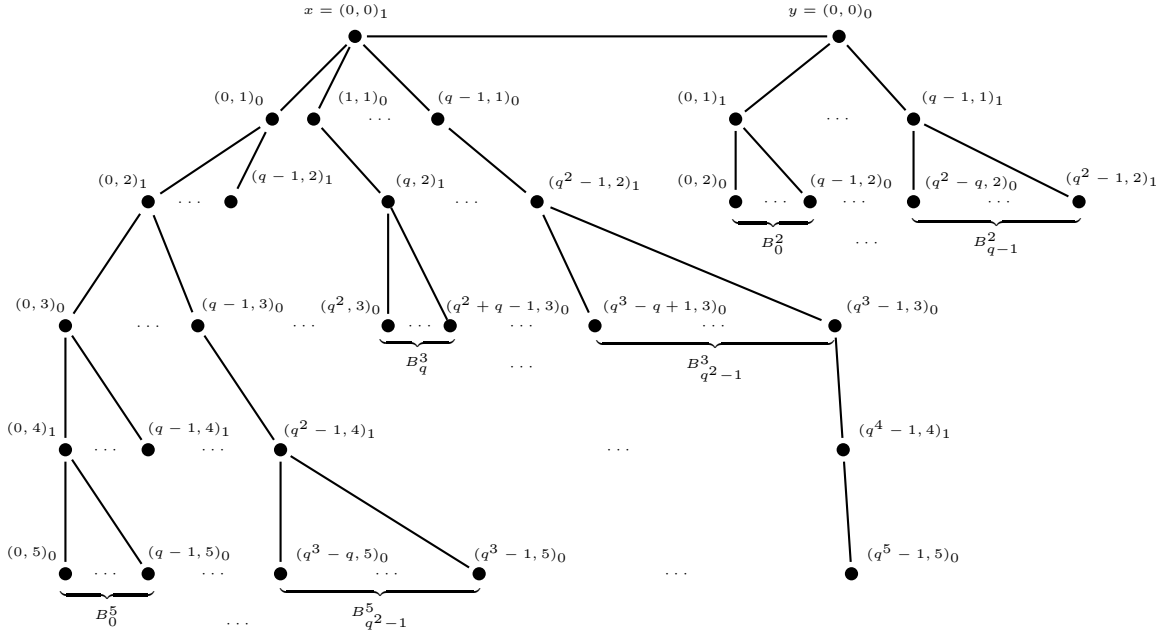


	\pspicture(-13.3,-4.3)(7.1,3.8)
	\centering
	\SpecialCoor
	\psset{linecolor=black, dotsize=5pt, unit=1.1}
	\rput(-.55,1){
		\dotnode(-7,2){000}
		\rput(-7.1,2.3){\tiny $x=(0,0)_1$}
		\dotnode(-1.15,2){001}
		\rput(-1.25,2.3){\tiny $y=(0,0)_0$}

		\ncline[nodesep=2pt]{001}{000}
		
		\dotnode(-8,1){011}
		\rput(-8.4,1.25){\tiny $(0,1)_0$}

		\dotnode(-7.5,1){211}
		\rput(-6.9,1.25){\tiny $(1,1)_0$}	
			\rput(-6.7,1){\tiny $\ldots$}		
		
		\dotnode(-6,1){111}
		\rput(-5.5,1.25){\tiny $(q-1,1)_0$}		
		
		\ncline[nodesep=2pt]{000}{011}
		\ncline[nodesep=2pt]{000}{111}
		\ncline[nodesep=2pt]{000}{211}
		
		\dotnode(-9.5,0){020}
		\rput(-9.8,0.25){\tiny $(0,2)_1$}
		
		\rput(-9,0){\tiny $\ldots$}
		
		\dotnode(-8.5,0){120}
		\rput(-7.75,0.3){\tiny$(q-1,2)_1$}	
		
		\dotnode(-6.6,0){q21}
		\rput(-6.2,0.25){\tiny$(q,2)_1$}

		\ncline[nodesep=2pt]{211}{q21}

		\rput(-5.65,0){\tiny $\ldots$}

		\dotnode(-4.8,0){420}
		\rput(-4.1,0.25){\tiny $(q^2-1,2)_1$}

		\ncline[nodesep=3pt]{011}{020}
		\ncline[nodesep=3pt]{011}{120}
		\ncline[nodesep=3pt]{011}{220}
		\ncline[nodesep=3pt]{111}{320}
		\ncline[nodesep=3pt]{111}{420}
		\ncline[nodesep=3pt]{111}{520}
		\ncline[nodesep=3pt]{211}{620}
		\ncline[nodesep=3pt]{211}{720}
		\ncline[nodesep=3pt]{211}{820}
		
		\dotnode(-10.5,-1.5){031}
		\rput(-10.8,-1.25){\tiny $(0,3)_0$}

\rput(-9.5,-1.5){\tiny $\ldots$}
\dotnode(-8.9,-1.5){231}
\rput(-8.35,-1.25){\tiny $(q-1,3)_0$}
		
		\ncline[nodesep=3pt]{020}{031}
		\ncline[nodesep=3pt]{020}{131}
		\ncline[nodesep=3pt]{020}{231}

		\ncline[nodesep=3pt]{220}{631}
		\ncline[nodesep=3pt]{220}{731}
		\ncline[nodesep=3pt]{220}{831}
		\rput(-7.6,-1.5 ){\tiny $\cdots$}
		\dotnode(-6.6,-1.5){930}
		\rput(-7.05,-1.25){\tiny $(q^2,3)_0$}
		
		\ncline[nodesep=2pt]{q21}{930}
	
	\rput(-6.2,-1.5){\tiny $\cdots$}
	
	\dotnode(-5.85,-1.5){q2+q-1}
	\ncline{q21}{q2+q-1}
	\rput(-5.15,-1.25){\tiny $(q^2 +q-1,3)_0$}
	
		\rput(-6.25,-1.85){\tiny $\underbrace{\ \ \ \ \ \ \ \ \ \ \ \ }_{B_q^3} $}
	\rput(-4.97,-1.5){\tiny $\cdots$}
		\dotnode(-4.1,-1.5){1331}
		\rput(-3.25,-1.25){\tiny $(q^3-q+1,3)_0$}
\rput(-2.65,-1.5){\tiny $\cdots$}
				\dotnode(-1.2,-1.5){1431}
		\rput(-0.5,-1.25){\tiny $(q^3-1,3)_0$}
		
		\ncline[nodesep=3pt]{420}{1231}
		\ncline[nodesep=3pt]{420}{1331}
		\ncline[nodesep=3pt]{420}{1431}
		
\rput(-2.65,-1.9){\tiny $\underbrace{ \ \ \ \ \ \ \ \ \ \ \ \ \ \ \ \ \ \ \ \ \ \ \ \ \ \ \ \ \ \ \ \ \ \ \ \ \ \ }_{B_{q^2-1}^3} $}

		\rput(-4.97,-2){\tiny $\cdots$}

		\dotnode(-10.5,-3){041}
		\rput(-10.9,-2.75){\tiny $(0,4)_1$}
		\rput(-10,-3){\tiny $\cdots$}
		\dotnode(-9.5,-3){841}
		\rput(-9,-2.75){\tiny $(q-1,4)_1$}
		
		\ncline[nodesep=2pt]{031}{041}
		\ncline[nodesep=2pt]{031}{841}

		\rput(-8.7,-3){\tiny $\cdots$}
		\dotnode(-7.9,-3){1741}
		\rput(-7.3,-2.75){\tiny $(q^2-1,4)_1$}
		\ncline[nodesep=2pt]{231}{1741}

	\rput(-3.8,-3){\tiny $\cdots$}
	
		\dotnode(-1.1,-3){4441}
		\rput(-.35,-2.75){\tiny  $(q^4-1,4)_1$}
				\ncline[nodesep=2pt]{1431}{4441}

	\dotnode(-10.5,-4.5){050}
\rput(-10.9,-4.25){\tiny $(0,5)_0$}
\rput(-10,-4.5){\tiny $\cdots$}
\dotnode(-9.5,-4.5){q-150}
\rput(-9,-4.25){\tiny $(q-1,5)_0$}

		\ncline[nodesep=2pt]{041}{050}
		\ncline[nodesep=2pt]{041}{q-150}
		
	\rput(-8.7,-4.5){\tiny $\cdots$}
		
		\dotnode(-7.9,-4.5){5450}
		\rput(-7.15,-4.25){\tiny $(q^3-q,5)_0$}
		\rput(-6.6,-4.5){\tiny $\cdots$}

		\ncline[nodesep=2pt]{1741}{5450}

		\dotnode(-5.5,-4.5){8150}
		\rput(-5.0,-4.25){\tiny $(q^3-1,5)_0$}
		
				\ncline[nodesep=2pt]{1741}{8150}

		\rput(-3.1,-4.5){\tiny $\cdots$}

		\dotnode(-1.,-4.5){13550}
		\rput(-.25,-4.25){\tiny $(q^5-1,5)_0$}
		
				\ncline[nodesep=2pt]{4441}{13550}
		
		\rput(-10,-4.9){\tiny $ \underbrace{\ \ \ \ \ \ \ \ \ \ \ \ \ \ \ }_{B_0^5} $}
		
		\rput(-8.4,-5.1){\tiny $ \cdots $}
		
		\rput(-6.7,-4.9){\tiny $ \underbrace{\ \ \ \ \   \ \ \ \ \ \ \ \ \ \ \ \ \ \ \ \ \ \ \ \ \ \ \ \ \  \ \ }_{B_{q^2-1}^5} $}
		
		\dotnode(-2.4,1){010}
		\rput(-2.8,1.25){\tiny $(0,1)_1$}

		\rput(-1.15,1){\tiny $\cdots$}
		
			\dotnode(-.25,1){110}
		\rput(.35,1.25){\tiny $(q-1,1)_1$}
		
		\ncline[nodesep=2pt]{001}{010}
		\ncline[nodesep=2pt]{001}{110}

				\dotnode(-2.4,0){021}
		\rput(-2.8,0.25){\tiny $(0,2)_0$}
		
		\rput(-1.9,0){\tiny $\cdots$}
		
		\dotnode(-1.5,0){121}
		\rput(-1.0,0.25){\tiny $(q-1,2)_0$}
		
		\rput(-1.95,-.4){\tiny $ \underbrace{\ \ \  \ \ \ \ \ \ \ \  \ \ }_{B_{0}^2} $}
		
		\ncline[nodesep=2pt]{021}{010}
		\ncline[nodesep=2pt]{121}{010}
		\rput(-.95,0){\tiny $\cdots$}
		
						\dotnode(-.25,0){221}
		\rput(.45,0.25){\tiny $(q^2-q,2)_0$}

		\rput(.75,-.4){\tiny $ \underbrace{\ \ \  \ \ \ \ \ \ \ \ \ \ \ \ \ \ \ \ \ \ \ \ \ \  \ \ }_{B_{q-1}^2} $}
		
		\rput(-.8,-.5){\tiny $\cdots$}
		
		\rput(.8,0){\tiny $\cdots$}
		\dotnode(1.75,0){321}
		\rput(2.15,0.3){\tiny $(q^2-1,2)_1$}
		
		\ncline[nodesep=2pt]{221}{110}
		\ncline[nodesep=2pt]{321}{110}
		
	}
	\endpspicture
	\caption{The labeled spanning tree of a $(q+1,12)$-Moore graph. }
	\label{FigArbol}
\end{figure}	
		
		The vertices of color $1$ are those labelled by $(i,j)_1$.		Assign color $2$ to vertex $(0,1)_0$. For every $i\in \{0, \dots, q^2-1\}$ the set $B^5_i$ contains at least $q-1\ge 3$ vertices that are not in ${\cal{O}}$. 
		Hence, we can choose two vertices of $B^5_i$ and color them with colors 2 and 3. For $a,b\in \{0,\dots , q^3-1\}$  with $|b-a|\ge q$, it follows that  $(a,5)_0\in B^5_i$ and $(b,5)_0\in B^5_j$ where $i\neq j$. 
		Therefore, the distance in $T$ between $(a,5)_0$ and $(b,5)_0$ is 4 or 6. And their distance in $G$ is at least  $4$.
		
		Similarly, for every $q\le i \le q^2-1$ we can choose two vertices of $B^3_i$ and color them with colors 2 and 3. 
		For $a,b\in \{q^2,\dots , q^3-1\}$ with $|b-a|\ge q$, it follows that   $(a,3)_0\in B^3_i$ and $(b,3)_0\in B^3_j$ where $i\neq j$. 
		Therefore the distance in $T$ between $(a,3)_0$ and $(b,3)_0$ is 4 or 6. And their distance in $G$ is at least  $4$.
		Moreover, the distance in $T$ between $(a,3)_0$ and $(b,5)_0$, for every $q^2\le a \le q^3-1$ and $0\le b \le q^3-1$ is $8$, and the distance between them in $G$ is at least $4$.
		Observe that for $0\le i \le q-1$,  $B^2_i\cap {\cal{O}}=\emptyset$.
		For every $0\le i \le q-1$ we can choose two vertices of $B^2_i$ and color them with colors 2 and 3. 
		Observe that for $a,b\in \{0,\dots , q^2-1\}$ such that $|b-a|\ge q$,  $(a,2)_0\in B^2_i$ and $(b,2)_0\in B^2_j$ with $i\neq j$. 
		Therefore the distance in $T$ between $(a,2)_0$ and $(b,2)_0$ is 4. And their distance in $G$ is $4$.
		Moreover, the distance in $T$ between $(a,2)_0$ and $(b,i)_0$, for $i\in \{1,3,5\}$ is 4,6 or 8, implying that they are at least at  distance $4$ in $G$.
		
		For every set $B^5_{iq}$ with $0\le i<q$, we can choose a vertex and assign  it  color  $4$. Observe that the distance between these vertices is $6$.
		
		Hence, for this  coloring we have $|\Gamma^{-1}(1)|=q^5+q^4+q^3+q^2+q+1$, 
		$|\Gamma^{-1}(2)|=q^3+q^2+1$,
		$|\Gamma^{-1}(3)|=q^3+q^2$,
		$|\Gamma^{-1}(4)|=q$,
		$|\Gamma^{-1}(5)|=q^3+1$  and the upper bound holds.
	\end{proof}

Finally, for the $(3,12)$-Moore graph and  the $(4,12)$-Moore graph, we computed an upper bound for the packing chromatic number using the backtracking function in the YAGS \cite{YAGS} package for GAP. 
For $(4,12)$-Moore graph we used the Adjacency matrix  from \cite{wwwdistreg}.
For  $(3,12)$-Moore graph we used the following definition in YAGS 

\noindent{\tiny 
g:=GraphByAdjacencies([ [ 64, 65, 78 ], [ 65, 66, 99 ], [ 66, 67, 84 ], [ 67, 68, 74 ], [ 68, 69, 95 ], [ 69, 70, 80 ], [ 70, 71, 76 ], [ 71, 72, 106 ], 
  [ 64, 72, 73 ], [ 73, 74, 87 ], [ 74, 75, 108 ], [ 75, 76, 93 ], 
  [ 76, 77, 83 ], [ 77, 78, 104 ], [ 78, 79, 89 ], [ 79, 80, 85 ], 
  [ 80, 81, 115 ], [ 73, 81, 82 ], [ 82, 83, 96 ], [ 83, 84, 117 ], 
  [ 84, 85, 102 ], [ 85, 86, 92 ], [ 86, 87, 113 ], [ 87, 88, 98 ], 
  [ 88, 89, 94 ], [ 89, 90, 124 ], [ 82, 90, 91 ], [ 91, 92, 105 ], 
  [ 92, 93, 126 ], [ 93, 94, 111 ], [ 94, 95, 101 ], [ 95, 96, 122 ], 
  [ 96, 97, 107 ], [ 97, 98, 103 ], [ 70, 98, 99 ], [ 91, 99, 100 ], 
  [ 100, 101, 114 ], [ 72, 101, 102 ], [ 102, 103, 120 ], [ 103, 104, 110 ], 
  [ 68, 104, 105 ], [ 105, 106, 116 ], [ 106, 107, 112 ], [ 79, 107, 108 ], 
  [ 100, 108, 109 ], [ 109, 110, 123 ], [ 81, 110, 111 ], [ 66, 111, 112 ], 
  [ 112, 113, 119 ], [ 77, 113, 114 ], [ 114, 115, 125 ], [ 115, 116, 121 ], 
  [ 88, 116, 117 ], [ 109, 117, 118 ], [ 69, 118, 119 ], [ 90, 119, 120 ], 
  [ 75, 120, 121 ], [ 65, 121, 122 ], [ 86, 122, 123 ], [ 71, 123, 124 ], 
  [ 67, 124, 125 ], [ 97, 125, 126 ], [ 64, 118, 126 ], [ 1, 9, 63 ], 
  [ 1, 2, 58 ], [ 2, 3, 48 ], [ 3, 4, 61 ], [ 4, 5, 41 ], [ 5, 6, 55 ], 
  [ 6, 7, 35 ], [ 7, 8, 60 ], [ 8, 9, 38 ], [ 9, 10, 18 ], [ 4, 10, 11 ], 
  [ 11, 12, 57 ], [ 7, 12, 13 ], [ 13, 14, 50 ], [ 1, 14, 15 ], 
  [ 15, 16, 44 ], [ 6, 16, 17 ], [ 17, 18, 47 ], [ 18, 19, 27 ], 
  [ 13, 19, 20 ], [ 3, 20, 21 ], [ 16, 21, 22 ], [ 22, 23, 59 ], 
  [ 10, 23, 24 ], [ 24, 25, 53 ], [ 15, 25, 26 ], [ 26, 27, 56 ], 
  [ 27, 28, 36 ], [ 22, 28, 29 ], [ 12, 29, 30 ], [ 25, 30, 31 ], 
  [ 5, 31, 32 ], [ 19, 32, 33 ], [ 33, 34, 62 ], [ 24, 34, 35 ], 
  [ 2, 35, 36 ], [ 36, 37, 45 ], [ 31, 37, 38 ], [ 21, 38, 39 ], 
  [ 34, 39, 40 ], [ 14, 40, 41 ], [ 28, 41, 42 ], [ 8, 42, 43 ], 
  [ 33, 43, 44 ], [ 11, 44, 45 ], [ 45, 46, 54 ], [ 40, 46, 47 ], 
  [ 30, 47, 48 ], [ 43, 48, 49 ], [ 23, 49, 50 ], [ 37, 50, 51 ], 
  [ 17, 51, 52 ], [ 42, 52, 53 ], [ 20, 53, 54 ], [ 54, 55, 63 ], 
  [ 49, 55, 56 ], [ 39, 56, 57 ], [ 52, 57, 58 ], [ 32, 58, 59 ], 
  [ 46, 59, 60 ], [ 26, 60, 61 ], [ 51, 61, 62 ], [ 29, 62, 63 ] ]);
}
  
\noindent and the following function for backtracking algorithm

{ 
\begin{verbatim}
L:=[];;
chk:=function(L,g)
local x,y;
if L=[] then return true; fi;
x:=Length(L);
for y in [1..x-1] do
if Distance(g,x,y)<=L[x] and L[x]=L[y] then
return false;
fi;
od;
return true;
end;
Backtrack(L,[1..26],chk,Order(g),g);
\end{verbatim}}
\noindent obtaining the following coloring, where the $i$th entry of $\Gamma$ corresponds to the color of the vertex $i$.   
{\tiny 
$\Gamma=[ 
1, 1, 1, 1, 1, 1, 1, 1, 1, 1, 1, 1, 1, 1, 1, 1, 1, 1, 1, 1, 1, 1, 1, 1, 1, 1, 1, 1, 1, 1, 1, 1, 1, 1,   1, 1, 1, 1, 1,1, 1, 1, 1, 1, 1, 1, 1, 1, 1, 1, 1, 1, 1, 1, 1, 1, 1, 1, 1, 1, 1, 1, 1,\\
2, 3, 2, 3, 2, 3, 2, 3, 4, 3, 5, 2, 6, 2, 5, 2, 4, 2, 7, 3,   8, 3, 2, 9, 2, 3, 2, 3, 5,
3, 10, 11, 2, 3, 12, 13, 2, 3, 2, 14, 3, 15, 2, 16, 3, 4, 17, 18, 3, 4, 19, 5, 3, 20, 21, 5, 3, 22, 23,\\ 24, 25, 2, 26 ]$},

\noindent implying that for the  $(3,12)$-Moore graph $G$,  $\chi_\rho (G)\le 26$. Similarly, for the $(4,12)$-Moore graph $G$, we obtained a packing coloring using 219 colors. Thus, for the  $(4,12)$-Moore graph $G$, $\chi_\rho (G)\le 219$.


\end{document}